\documentclass[final,5p,times,twocolumn]{elsarticle}
\usepackage{tikz}
\usetikzlibrary{shapes,snakes}
\usetikzlibrary{arrows,positioning}
\usetikzlibrary{decorations.markings}
\usetikzlibrary{shapes.geometric,calc}
\usetikzlibrary{backgrounds,automata}
\usepackage{multicol, blindtext}
\usepackage{multirow}
\usepackage{wrapfig}
\usepackage{pgfplots}
\usepackage{tkz-fct}
\usepackage{icomma}
\usepackage{graphicx}
\usepackage{algorithm}
\usepackage{algorithmic}
\usepackage{graphics}
\usepackage{subfigure}
\usepackage{hyperref}
\usepackage{bbm}
\usepackage{cmll}
\usepackage{mathdots}
\usepackage{graphicx}
\usepackage{epsfig}
\usepackage{3dplot}
\usepackage{amssymb}
\usepackage{amsthm, amsmath}
\newtheorem{theorem}{Theorem}

\newtheorem{example}{Example}
\newtheorem{remark}{Remark}
\theoremstyle{definition}

\journal{}

\begin{document}

\begin{frontmatter}


\author[label1]{Jinwook Lee \corref{cor1}}
\ead{jl3539@drexel.edu}
\address[label1]{Decision Sciences and MIS, LeBow College of Business, Drexel University, Philadelphia, PA 19104, United States}
\cortext[cor1]{Corresponding author, Phone: 1-215-571-4672, Fax: 1-215-895-2907}

\title{On the probability of Boolean functions of events in the $n$-dimensional Euclidean space}


\begin{abstract}
It is shown in \cite{union2} and \cite{union1} that for the union of $N$ orthants in $R^n$ ($N \gg n$) there exists an efficient and systematic way to find the exact value, using a suitable partial order relation construction. In this paper our events are  hyperrectangles (or $n$-orthotopes), the Cartesian product of intervals -- another important sets (or events) in both theory and practice. We have discovered a new efficient algorithm for the union of such events. With other important Boolean functions we present optimization problem formulations for both hyperreectangles and orthants.    
 \end{abstract}


\end{frontmatter}



\section{Introduction}

Let $A_1, \dots, A_N$ be events in an arbitrary probability space. In many applications we are interested to find probabilities of Boolean functions of them, among which prominent are the following:
\begin{center}
\begin{itemize}
\item [-] at least one of them occurs;
\item [-] at least $r$ of them occur;
\item [-] exactly $r$ of them occur.
\end{itemize}
\end{center}
In reliability theory, where consecutive events play important role, further Boolean functions of events can be used. Our results can be applied for probabilities of these other Boolean functions, as well, but here we restrict ourselves to the above-mentioned three.

Given such Boolean functions of events, related probabilities and conditional expectations are useful in decision making. For example, If there are $N$ divisions of a business organization, then a massive failure of just a few of the divisions will possibly collapse the entire organization. For this reason, a problem of ``at least $r$ ($1 \leq r \ll N$)" events comes into play for suitable decision making, using the probabilities of such events and their related conditional expectations (e.g., probability weighted sum of losses (or profits)), etc.

Our events in this paper are hyperrectangles (or $n$-orthotopes) in the $n$-space, designated by 
\begin{equation}\label{notation}
A_1 = A(z^{(1)}), \dots, A_N = A(z^{(N)}), 
\end{equation}
where $z^{(1)}, \dots, z^{(N)}$ are the pairs of lower and upper vertices of hyperrectangles $R^n$.
$z^{(i)} = [l^{(i)}, u^{(i)}]$ where $l^{(i)}, u^{(i)}$ for $i \in R^n, i = 1, \dots, N$.
Each of the $z^{(i)}$'s is a pair of vertices.

For simplicity, let us consider a cube in $R^3$: 
$$z^{(1)} = [(1, 2, 3), (4, 5, 6)].$$ 
Then our event is:
\begin{equation}\label{1.2}
\begin{array}{lcl}
A(z^{(1)}) &=& A([(1, 2, 3), (4, 5, 6)])\\
&=& \{(z_1, z_2, z_3) \ | \ (1, 2, 3) \leq (z_1, z_2, z_3) \leq (4, 5, 6) \},
\end{array}
\end{equation}
Note that there are eight vertices for the cube $z^{(1)}$, but we only need a pair. It is easy to see that no matter what dimension we have, say $n$ some large positive integer, we only need a pair (i.e., lower and upper corners) out of $2^n$ vertices. 

 (\ref{1.2}) can also be represented by the Cartesian product:
\begin{equation}\label{1.3}
A(z^{(1)}) = Z^{(1)} = Z_1^{(1)} \times Z_2^{(1)} \times Z_3^{(1)},
\end{equation}
where $Z_1 = \{z_1 \ | \ 1 \leq z_1 \leq 4\}$, $Z_2 = \{z_2 \ | \ 2 \leq z_2 \leq 5\}$ and $Z_3 = \{z_3 \ | \ 3 \leq z_3 \leq 6\}$.
Equations (\ref{1.2}) and (\ref{1.3}) are equivalent to
\begin{equation}\label{1.4}
\{(z_1, z_2, z_3) \ | \ z_1 \in [1, 4] \wedge z_2 \in [2, 5] \wedge z_3 \in [3, 6] \}.
\end{equation}

Let $f$ denote the p.d.f. of a random vector $Z \in R^3$. Then we can write
\begin{equation}\label{1.5}
P(A(z^{(1)})) = \displaystyle \int_1^4   \int_2^5   \int_3^6   f(z_1, z_2, z_3) dz_3 dz_2 dz_1.
\end{equation}

\section{Vertex comparison for empty intersections of hyperrectangles}
\begin{theorem}\label{thm1}
A pairwise intersection of $N$ hyperrectangles $A_1, \dots, A_N$ in $R^n$ is not empty if the condition (i) or (ii) is met.

\begin{itemize}
\item [(i)] For every $k = 1, \dots, n$, we have: 
$$\max\{l_k^{(i)}, l_k^{(j)}\}  \leq  \min\{u_k^{(i)}, u_k^{(j)}\},$$
where integers $i, j$ such that $1 \leq i < j \leq N.$
\item [(ii)] In case of a continuous distribution, for every $k = 1, \dots, n$, we have: 
$$\max\{l_k^{(i)}, l_k^{(j)}\}  <  \min\{u_k^{(i)}, u_k^{(j)}\},$$
where integers $i, j$ such that $1 \leq i < j \leq N.$
\end{itemize}
\end{theorem}

\begin{proof}
Suppose not. Then there exists $k$ such that $$\max\{l_k^{(i)}, l_k^{(j)}\}  >  \min\{u_k^{(i)}, u_k^{(j)}\},$$ which can't make a nonempty intersection. 
\end{proof}

\begin{remark}[Screening test]
Theorem \ref{thm1} can be used as a screening test. We can remove empty intersections from the inclusion-exclusion formula.
\end{remark}

\begin{example}\label{ex1}
For simplicity let us consider a continuous uniform random vector $(X_1, X_2)^T$ with support set $\{(x_1, x_2) \ | \ 0 \leq x_i \leq 10, i=1,2   \}$.
Suppose $A_1 = [(5,6), (9,9)]$, $A_2 = [(2, 4), (6, 7)]$, $A_3 =[(1, 3), (4, 8)]$, $A_4 = [(3, 2), (7,10)]$, $A_5=[(0,1), (10,5)]$. Note that they are depicted in Figure \ref{fig1}.
\end{example}

In Example \ref{ex1} there are 10 pairs $A_iA_j$, $i < j,$ for $i.j = 1, \dots, 5$ and their emptyness testing is as follows.

\begin{equation}\label{1.6}
\begin{array}{cc}
\text{\underline{Pairs}} & \text{\underline{Are conditions of Theorem \ref{thm1} met?}}\\
A_1A_2 = [(5, 6), (6, 7)]&\text{yes}\\
A_1A_3 = [(5, 6), (4, 8)]& \text{no good}\\
A_1A_4 = [(5, 6), (7, 9)]& \text{yes}\\
A_1A_5 = [(5, 6), (9, 5)]& \text{no good}\\
A_2A_3 = [(2, 4), (4, 7)]& \text{yes}\\
A_2A_4 = [(3, 4), (6, 7)]& \text{yes}\\
A_2A_5 = [(2, 4), (6, 5)]& \text{yes}\\
A_3A_4 = [(3, 3), (4, 8)]& \text{yes}\\
A_3A_5 = [(1, 3), (4, 5)]& \text{yes}\\
A_4A_5 = [(3, 2), (7, 5)]& \text{yes}\\
\end{array}
\end{equation}
We can draw Figure \ref{fig2} based on (\ref{1.6}) which shows $A_1A_3$ and $A_1A_5$ are empty.
In Figure \ref{fig2} there are eight 2-vertex clique, five 3-vertex clique and one 4-vertex clique.
In Example \ref{ex1} there are 10 triples $A_iA_jA_k$, $i < j < k,$ for $i.j, k = 1, \dots, 5$ where there might be some empty ones. 
Using the information of (\ref{1.6}) we can remove empty triples including the empty pairs of $A_1, A_3$ and $A_1, A_5.$
Therefore, we remove the following triples: $A_1A_2A_3, A_1A_3A_4, A_1A_3A_5$ and $A_1A_2A_5, A_1A_3A_5, A_1A_4A_5$ which are empty.

\begin{equation}\label{1.7}
\begin{array}{cc}
\text{\underline{Triples}} & \text{\underline{Are conditions of Theorem \ref{thm1} met?}}\\
A_1A_2A_4 = [(5, 6), (6, 7)]& \text{yes}\\
A_2A_3A_4 = [(3, 4), (4, 7)]& \text{yes}\\
A_2A_3A_5 = [(2, 4), (4, 5)]& \text{yes}\\
A_2A_4A_5 = [(3, 4), (6, 5)]& \text{yes}\\
A_3A_4A_5 = [(3, 3), (4, 5)]& \text{yes}\\
\end{array}
\end{equation}

Using the information of (\ref{1.7}) we find out that there is only one 4 tuple $A_2A_3A_4A_5$ and no 5 tuple.
\begin{equation}\label{1.8}
\begin{array}{cc}
\text{\underline{4-tuple}} & \text{\underline{Are conditions of Theorem \ref{thm1} met?}}\\
A_2A_3A_4A_5 = [(3, 4), (4, 5)] & \text{yes}
\end{array}
\end{equation}
With (\ref{1.6}), (\ref{1.7}), (\ref{1.7}), we can write up an inclusion exclusion formula:
$S_1 - S'_2 + S'_3 - S'_4,$ where we have only 19 terms compared to 31 terms in the original inclusion-exclusion formula.

\begin{figure}[ht!]
\begin{center}
\includegraphics[width=0.5\textwidth]{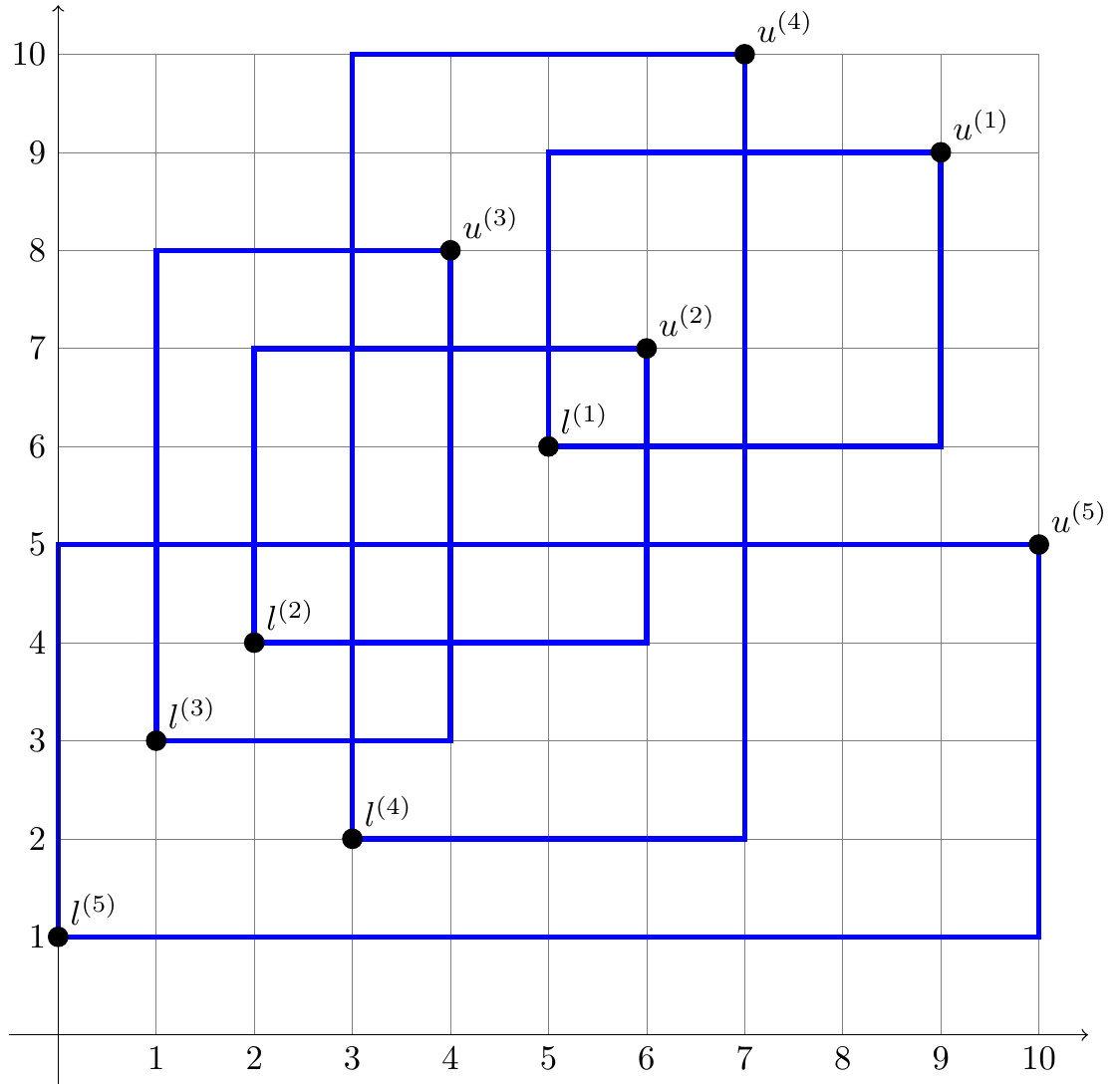}
\end{center}
  	  \caption{Description of Example 1.}\label{fig1}
         \end{figure}

\begin{figure}[ht!]
\begin{center}
\includegraphics[width=0.5\textwidth]{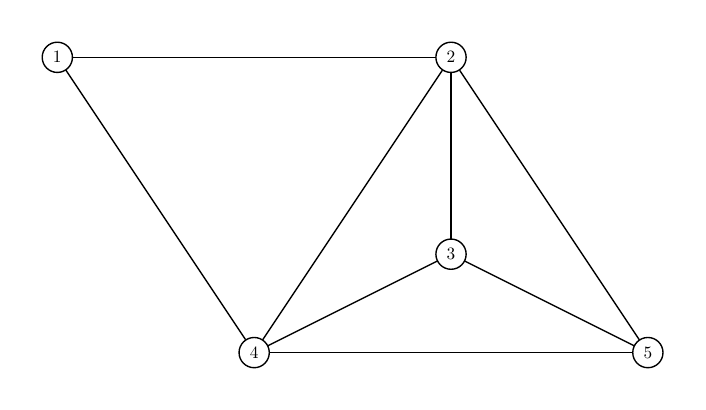}
\end{center}
  	  \caption{A graph with nonempty intersections. The clique number of the graph is 4.}\label{fig2}
         \end{figure}

\begin{example}\label{ex2}
For simplicity let us consider a continuous uniform random vector $(X_1, X_2, X_3)^T$ with support set $\{(x_1, x_2, x_3) \ | \ 0 \leq x_i \leq 5, i=1,2, 3 \}$.
Suppose $A_1 = [(0, 0, 0), (2, 2, 2)]$, $A_2 = [(3, 1, 3), (5, 3, 5)]$, $A_3 =[(1, 3, 3), (3, 5, 5)]$, $A_4 = [(4, 4, 4), (5, 5, 5)]$, $A_5=[(2, 2, 2), (3, 3, 4)]$, $A_6 = [(1, 4, 1), (2, 5, 2)]$, $A_7 = [(4, 1, 4), (5, 2, 5)]$.
\end{example}

In Example \ref{ex1} there are 21 pairs $A_iA_j$, $i < j,$ for $i.j = 1, \dots, 7$ and their emptyness testing is as follows.

\begin{equation}\label{1.9}
\begin{array}{cc}
\text{\underline{Pairs}} & \text{\underline{Are conditions of Theorem \ref{thm1} met?}}\\
A_1A_2 = [(3, 1, 3), (2, 2, 2)]&\text{no good}\\
A_1A_3 = [(1, 3, 3), (2, 2, 2)]& \text{no good}\\
A_1A_4 = [(4, 4, 4), (2, 2, 2)]& \text{no good}\\
A_1A_5 = [(2, 2, 2), (2, 2, 2)]& \text{no good}\\
A_1A_6 = [(1, 4, 1), (2, 2, 2)]&\text{no good}\\
A_1A_7 = [(4, 1, 4), (2, 2, 2)]&\text{no good}\\
A_2A_3 = [(3, 3, 3), (3, 3, 5)]& \text{no good}\\
A_2A_4 = [(4, 4, 4), (5, 3, 5)]& \text{no good}\\
A_2A_5 = [(3, 2, 3), (3, 3, 4)]& \text{no good}\\
A_2A_6 = [(3, 4, 3), (2, 3, 2)]&\text{no good}\\
A_2A_7 = [(4, 1, 4), (5, 2, 5)]&\text{yes}\\
A_3A_4 = [(4, 4, 4), (3, 5, 5)]&\text{no good}\\
A_3A_5 = [(2, 3, 3), (3, 3, 4)]& \text{no good}\\
A_3A_6 = [(1, 4, 3), (2, 5, 2)]&\text{no good}\\
A_3A_7 = [(4, 3, 4), (3, 2, 5)]&\text{no good}\\
A_4A_5 = [(4, 4, 4), (3, 3, 4)]&\text{no good}\\
A_4A_6 = [(4, 4, 4), (2, 5, 2)]&\text{no good}\\
A_4A_7 = [(4, 4, 4), (5, 2, 5)]&\text{no good}\\
A_5A_6 = [(2, 4, 2), (2, 3, 2)]&\text{no good}\\
A_5A_7 = [(4, 2, 4), (3, 2, 4)]&\text{no good}\\
A_6A_7 = [(4, 4, 4), (2, 2, 2)]&\text{no good}\\
\end{array}
\end{equation}

It is shown in (\ref{1.9}) that there is only one nonempty pair $A_2A_7$ which can be depicted as in Figure \ref{fig3}.
And the probability of union can be calculated only by $S_1 - S_2 = \sum_{i=1}^7 P(A_i) - P(A_2A_7) $ which has only 8 terms instead of 127 terms.

\begin{figure}[ht!]
\begin{center}
\includegraphics[width=0.5\textwidth]{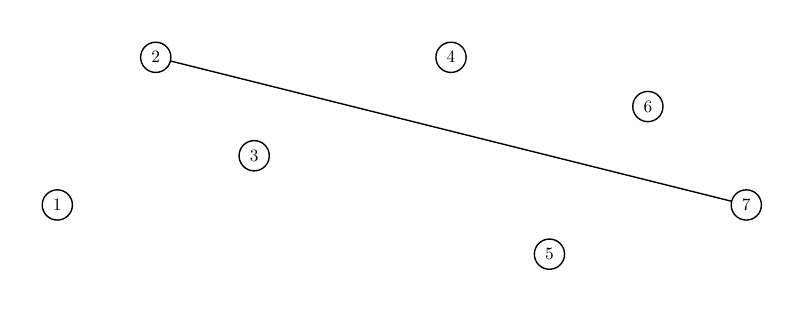}
\end{center}
  	  \caption{A graph with nonempty intersections. The clique number of the graph is 2.}\label{fig3}
         \end{figure}

\section{Probability bounding using linear program}

\cite{bounding1}, \cite{bounding2} and \cite{bounding3} observed that the sharp probability bounds, using $S_1, \dots, S_m$, are optimum values of LP's that he called binomial moment problems. In doing so, he opened a new research area: the discrete moment problems. The term binomial moment comes from the fact that if $\xi$ is the number of events in $A_1, \dots, A_N$, which occur, then
\begin{equation}\label{0.4}
S_k = E\left[  \binom{\xi}{k} \right], k =1, \dots, N.
\end{equation}
By convention we write $S_0 = 1$ with which (\ref{0.4}) holds also for $k=0$. Equation (\ref{0.4}) is a classical theorem and it is not known who proved it first.

Starting from (\ref{0.4}) and introducing the notations
\begin{equation}
p_i = P(\xi = i), i=0, 1, \dots, N,
\end{equation}

The binomial problems for:
\begin{itemize}
\item [-] at least one of them occurs;
\item [-] at least $r$ of them occur;
\item [-] exactly $r$ of them occur.
\end{itemize}
can be formulated as the following.

\begin{equation}\label{1.12}
\begin{array}{ll}
\min (\max) & \displaystyle \sum_{i=1}^N p_i \\
\text{subject to} & \displaystyle  \sum_{i=0}^N \binom{i}{k} p_i = S_k, k=0, 1, \dots, m\\ [2ex]
& p_i \geq 0, i=0, 1, \dots, N,\\
\end{array}
\end{equation}

\begin{equation}\label{1.13}
\begin{array}{ll}
\min (\max) & \displaystyle \sum_{i=r}^N p_i \\
\text{subject to} & \displaystyle  \sum_{i=0}^N \binom{i}{k} p_i = S_k, k=0, 1, \dots, m\\ [2ex]
& p_i \geq 0, i=0, 1, \dots, N,\\
\end{array}
\end{equation}

\begin{equation}\label{1.14}
\begin{array}{ll}
\min (\max) & \displaystyle  p_r \\
\text{subject to} & \displaystyle  \sum_{i=0}^N \binom{i}{k} p_i = S_k, k=0, 1, \dots, m\\ [2ex]
& p_i \geq 0, i=0, 1, \dots, N.
\end{array}
\end{equation}

Instead of problem (\ref{1.12}) \cite{bounding1} used a simpler problem, removing $p_0$ and the first constraint:
\begin{equation}\label{1.15}
\begin{array}{ll}
\min (\max) & \displaystyle \sum_{i=1}^N p_i \\
\text{subject to} & \displaystyle  \sum_{i=1}^N \binom{i}{k} p_i = S_k, k=1, \dots, m\\ [2ex]
& p_i \geq 0, i= 1, \dots, N,\\
\end{array}
\end{equation}

\cite{bounding1}  characterized the dual feasible bases and showed that there can be used closed form formulas for $m =2,3.$
For the case of a general $m$ he gave elegant dual algorithms to obtain bounds for the probabilities in question. 
The results have been used by \cite{bounding0} to obtain closed form formulas for the case $m=4$ and also for the case of a general $m$.

Another direction in probability bounding problems is to use the individual probabilities $p_{i_1 \dots i_k} = P(A_{i_1}\dots A_{i_k}), k \leq m,$ rather than the aggregated values $S_1, \dots, S_m.$ \cite{bool1} formulated probability bounding problems using individual probabilities that we call today LP's and further developed it in (\citeyear{bool2}). 
It was \cite{hailperin} who formulated the general LP that we call today Boolean probability bounding problem. In order to present it in a simple  form we introduce the notations: 
\begin{equation}\nonumber
\begin{array}{rcl}
p_I &     = &\displaystyle P \Big( \bigcap_{j \in I} A_j \Big), \ I,J \subset  \{1, \dots, N\},\\
a_{IJ}  & = & \left\{ \begin{array}{l} 1, \text{ if }  I \subset J \\
			                      0,  \text{ if } I \nsubseteq  J, \ I, J \subset \{1, \dots, N\} 
			                      \end{array}\right. ,\\ [2ex]
x_J &    =  &\displaystyle P\Big( \big( \bigcap_{j \in J} A_j \big) \cap \big(\bigcap_{j \notin J} \bar{A}_j \big) \Big),\\ [2ex]
\end{array}
\end{equation}
The probability $p_I$ means that all events $A_j, \ j \in K$ occur and the probability $x_J$ means that all events $A_j, \ j\in I$ occur but the other do not occur.
The Boolean lower and upper bounding problems for the union are the following:
\begin{equation}\nonumber
\begin{array}{l}
\min (\max) \displaystyle \sum_{\emptyset \neq J \subset \{1, \dots, N \}} x_J\\
\text{subject to}\\[1ex]
\displaystyle \sum_{J \subset \{1, \dots, N \} } a_{IJ}x_J = p_I, \ I \subset \{1, \dots, N \} , |I| \leq m\\[3ex]
x_J \geq 0, J  \subset \{1, \dots, N \}.
\end{array}
\end{equation}
A counterpart of it also exist, where we use $I, J \neq \emptyset$.
The first closed form bounding formula was given by \cite{hunter} and \cite{worsley}. In this connection we also mention \cite{92prekopa}, \cite{97caen},  \cite{kuai}, \cite{cherry}, \cite{polycom}. \cite{92prekopa} has shown that the Hunter-Worsley bounds in the objective function value corresponding to a dual feasible basis in the modified Boolean minimization problem.
\cite{cherry} have shown the same about their upper bound.

\section{The binomial moment problem formulations for the probabilities that at least $k$ and exactly $k$ events occur}
In case of orthants, the bounds of the probability that ``at least $k$ events occur" can be found by the following
\begin{equation}
\begin{array}{ll}
\displaystyle  \min (\max) &  \displaystyle   \sum_{i=k}^N p_i\\[2ex]
\displaystyle  & \displaystyle   \sum_{i=1}^N p_i = Q\\
& \displaystyle   \sum_{i=1}^N i p_i = S_1\\
& \displaystyle   \sum_{i=1}^N \binom{i}{2} p_i = S_2\\
& \displaystyle   \sum_{i=1}^N  \binom{i}{3} p_i = S_3\\
\displaystyle  & p_i \geq 0,  i=1, \dots, N,  
\end{array}
\end{equation}
where $Q$ denotes the probability of union that can be calculated by by Theorem 9 of \cite{union2}. The worst case time complexity of calculating $Q$ is  $O(n^2N^2 + n \cdot N \log N)$ for $R^n, n\geq 3$ ($O(n^2N^2)$ for sorted sets). Note that in practice we typically have $N \gg n$. If we assume that $N \geq n^2$, then it can be written $O(N^3)$ for sorted sets.
In this case, the above formulation would be more efficient and give better bounds than the original binomial moment problem formulation. 
In the bivariate case $Q$ can be calculated in a constant time for sorted sets since we have $Q = S_1-S_2'$, where $S'_2$ denotes the sum of incomparable pairwise intersections (see \cite{union1} for more details.).

Similarly, the bounds of the probability that ``exactly $k$ events occur" can be found by the following
\begin{equation}
\begin{array}{ll}
\min (\max) & p_k\\
\text{subject to} & \displaystyle   \sum_{i=1}^N p_i = Q\\
& \displaystyle   \sum_{i=1}^N i p_i = S_1\\
& \displaystyle   \sum_{i=1}^N \binom{i}{2} p_i = S_2\\
& \displaystyle   \sum_{i=1}^N  \binom{i}{3} p_i = S_3\\
& p_i \geq 0, i=1, \dots, N.  

\end{array}
\end{equation}

In case of hyperrectangles, we can just use (\ref{1.13}), (\ref{1.14}) and (\ref{1.15}) with a lot less terms.

\bibliographystyle{elsarticle-harv}
\bibliography{Nov_17_refs}







\end{document}